\documentclass[11pt]{amsart}
\usepackage{amsfonts}
\usepackage{amsmath}
\usepackage{amssymb}
\usepackage{amscd}
\usepackage{graphicx}
\usepackage{fancyhdr}
\usepackage{color}
\usepackage{graphicx}
\usepackage{gloss}

\usepackage{mathtools}               

\usepackage[normalem]{ulem}

\usepackage{lipsum}                         

\usepackage{mathrsfs}              

\usepackage{tikz-cd}           

\usepackage{float}

\usepackage{extarrows}          

\input xy
\xyoption{all}

\newtheorem{theorem}{Theorem}

\newtheorem*{main theorem}{Main Theorem}

\newtheorem{definition}[theorem]{Definition}

\newtheorem{lemma}[theorem]{Lemma}

\newtheorem{proposition}[theorem]{Proposition}
\newtheorem{remark}[theorem]{Remark}

\begin{document}

\title{Simplicial volume of fiber bundles with nonpositively curved fibers}

\author{Xiaofeng Meng}
\address{School of Mathematical Sciences, Fudan University, Shanghai 200433, China}
\email{xfmeng17@fudan.edu.cn}


\date{}

\dedicatory{}

\commby{}

\begin{abstract}
We prove the simplicial volume of the total space of a smooth fiber bundle with fiber being an oriented closed connected (occ)
manifold of nonpositive curvature and negative Ricci curvature over an occ manifold with a closed universal covering is zero.
Furthermore, if the fiber is an occ negatively curved manifold with dimension more than $2$, the simplicial volume of the total space is zero if and only if the simplicial volume of the base space is zero.
\end{abstract}
\maketitle
\section{Introduction}
The simplicial volume is also called the Gromov norm in literature.
It is a topological invariant of manifolds introduced by Thurston \cite{thurston1978geometry}
and Gromov \cite{gromov1982volume}. For an oriented closed connected $n$-dimensional manifold $M$, its
simplicial volume is defined as
\[
\|M\|=\mathrm{inf} \left \{ \sum_{i}|a_{i}|\ \middle | \  \bigl[\sum_{i}a_{i}\sigma_{i}\bigr]=[M]\in
H_{n} (M,\mathbb{R}  ) \right \},
\]
where $[M]\in H_n(M;\mathbb R)$ is the fundamental class of $M$ with real coefficients.

For the product $M\times N$ of two closed manifolds $M$ and $N$ with dimensions $m$ and $n$ respectively,
Gromov \cite{gromov1982volume} proved the relation:
\begin{equation}\label{2021052101}
    \|M\|\cdot\|N\| \leq \|M\times N\| \leq \binom{n+m}{n}\|M\|\cdot\|N\|.
\end{equation}
For a nontrivial fiber  bundle, Gromov
proved in \cite{gromov1982volume} that for a smooth fiber bundle whose fiber and base are occ manifolds, the total space has zero
simplicial volume if the fiber is amenable, i.e., its fundamental group is an amenable group.
In some very special cases, there is a relation between the simplicial volume of the total space
and the product of the simplicial volumes of fiber and base
proved by Hoster in \cite{hoster2001simplicial} and Bucher in \cite{bucher2009simplicial}.

However, one cannot expect a relation as (\ref{2021052101}) for a  fiber bundle in general.
Hoster and Kotschick \cite{hoster2001simplicial} gave an example where a closed hyperbolic $3$-manifold is a fiber bundle over the circle.
It is well-known that the simplicial volume of a closed hyperbolic manifold is non-zero (see \cite{gromov1982volume} or \cite{thurston1978geometry}), while the simplicial volume of the cicle is zero (by \cite{gromov1982volume} the simplicial volume of an amenable occ manifold is zero). 

Nevertheless, we can still get some results by making assumptions on the base space.
For example,  L{\"o}h and Moraschini proved vanishing results for certain mapping tori in \cite{lm21}. 
Kastenholz and Reinhold \cite{kr21} proved vanishing results when the base space is $S^d$ with $d\geq 2$ and when the base space is a flexible 2-connected manifold with additional assumptions on fiber.

In this paper, we use the methods in \cite{bfj16} and \cite{fg16} to give several vanishing and non-vanishing results with fibers being occ negatively curved manifolds and some more generalized results.

Let
\begin{equation*}
    \begin{CD}
M @>>>E \\
@. @VpVV \\
@. B
    \end{CD}
\end{equation*}
be a smooth fiber bundle over an occ manifold with occ fiber.
If $M$ is equipped with a Riemannian metric of nonpositive sectional curvature and negative Ricci curvature, we have the following theorem.
\begin{theorem}\label{thm:ricci<0}
Let $M$ be an occ curved manifold with $K\leq 0$ and $Ricci<0$. 
Let $E$ be a smooth fiber bundle over an occ manifold $B$ with fiber $M$.
If the universal covering of $B$ is closed.
Then $\|E\|=0$.
\end{theorem}
We have the following more general Theorem.
\begin{theorem}\label{thm:uc of b closed}
Let $M$ be an occ aspherical manifold with the center of its fundamental group
$Center(\pi_1(M))=0$. 
Let $E$ be a smooth fiber bundle over an occ manifold $B$ with fiber $M$.
If the universal covering of $B$ is closed.
Then $\|E\|=0$.
\end{theorem}
If $M$ is equipped with a negatively curved metric, we have the following theorem.
\begin{theorem}
Let $M$ be an ooc negatively curved manifold with dimension more than $2$. 
Let $E$ be a fiber bundle over an ooc manifold $B$ with fiber $M$.
Then $\|E\|=0$ if and only if $\|B\|=0$.
\end{theorem}
It follows directly from the following Theorem.
\begin{theorem}\label{thm:all B}
Let $p:E\rightarrow B$ be an oriented smooth fiber bundle over an occ manifold $B$ with fiber $M$ being an occ manifold. 
If $M$ satisfies the following conditions
\begin{itemize}
    \item[(a)] $dim(M)\geq 3$;
    \item[(b)] $M$ is nonpositively curved closed manifold;
    \item[(c)] $Out(\pi_1(M))$ is finite;
    \item[(d)] $Center(\pi_1(M))=1$.
\end{itemize}
Then $\|E\|=0$ if and only if $\|M\|\cdot\|B\|=0$.
\end{theorem}
\begin{remark}
Same as in \cite{bfj16}, the following two categories of manifolds satisfies assumption (a)-(d) above.
\begin{itemize}
    \item [a)]$M$ is an occ negatively curved manifolds with dimension at least 3.
    \item[b)] $M$ is a nonpositively curved locally symmetric space of noncompact type, such that it has no finite sheet cover $\hat{M}=A\times B$ where $dim(A)=2$. 
\end{itemize}
\end{remark}

We arrange the rest of this paper as follows. In Section 2 we first recall the center Theroem, the definition of fiber homotopocally trivial and a feature about CW-complex. Then we prove Theorems \ref{thm:ricci<0} and \ref{thm:uc of b closed}.
In Section 3 we first recall the classification of smooth fiber bundles. Then we prove Theorem \ref{thm:all B} by using certain pull-back bundles.

\subsection*{Acknowledgements}
The author would like to express her deep gratitude to her advisor Jixiang Fu for constant support.

\section{Proof of Theorems \ref{thm:ricci<0} and \ref{thm:uc of b closed}}
Theorem \ref{thm:uc of b closed} is proved by applying method developed in the proof of \cite[Proposition1.4]{fg16}.

Recall The center Theorem in Lawson and Yau \cite{ly72}.
\begin{proposition}[The center Theorem]
Let $M$ be an occ manifold with $K\leq 0$.
Then $Center(\pi_1(M))=k\mathbb Z$ for some $k\geq 0$
and there exists an covering 
$T^k\times M'\rightarrow M$,
where $M'$ is an occ $K\leq 0$ manifold.
\end{proposition}
Therefore, Theorem \ref{thm:ricci<0} is a straight forward corollary of Theorem \ref{thm:uc of b closed}.
Recall that each closed manifold admits a triangulation, i.e., a homeomorphism with the support of a finite simplicial complex according to \cite{c35} and \cite{w40}.
Note that a finite simplicial complex admits a nature CW-complex structure.
\begin{definition}
Let $M\rightarrow E\stackrel{p}{\rightarrow}B$ be a smooth bundle. 
We call it fiber homotopically trivial if there exists a continuous map
$q:E\rightarrow M$
such that 
$q|_{M_x}:M_x\rightarrow M$
is a homotopy equivalence for all $x\in B$
\end{definition}
\begin{remark}
Note that according to \cite{dl59}, if $M\rightarrow E\stackrel{p}{\rightarrow}B$ is fiber homotopically trivial with $M$ and $B$ being occ manifolds.
Then $E$ is homotopy equivalent to $M\times B$.
\end{remark}

\begin{proof}[Proof of Theorem \ref{thm:uc of b closed}]
Denote the universal covering of $B$ by $\pi:\tilde B\rightarrow B$.
Consider the pullback bundle
\begin{equation*}
    \begin{CD}
    \pi^\ast E @>>>E\\
    @VVV @VVV\\
    \tilde B @>\pi>>B.    
    \end{CD}
\end{equation*}
With $\tilde B$ being closed, $\pi:\tilde B\rightarrow B$ is a finite-sheeted covering.
Hence $\pi^\ast E\rightarrow E$ is also a finite-sheeted covering.
Since simplicial volume is multiplicative with respect to finite coverings, 
we can assume that $B$ is simply connected.
As $B$ is a closed manfiold, there is a finite simplicial complex $K$ and a homeomorphism $r:K\rightarrow B$.
Consider the pullback bundle 
\begin{equation*}
    \begin{CD}
    r^\ast E @>>>E\\
    @VVV @VVV\\
    K @>r>>B.
    \end{CD}
\end{equation*}
Then $r^\ast E$ is homeomorphism to $E$.
As simplicial volume is a topological invariant, we consider $r^\ast E$ from now on.

Let $K^1$ be the $1$-skeleton of K. 
Denote the inclusion map by $\sigma:K^1\rightarrow K$.
With $K$ being simply connected, the inclusion map
$\sigma:K^1\rightarrow K$ is homotopic to a constant map 
$c:K^1\rightarrow \{x_0\}\subset K$ through map $H:K^1\times [0,1]\rightarrow K$.
With $r^\ast E\rightarrow K$ being a fibration, it satisfies the covering homotopy theorem.
Hence $H$ is covered by a map 
$\tilde{H}:p^{-1}(K^1)\times [0,1]\rightarrow r^\ast E$, 
with 
$\tilde{H}(\cdot,0):p^{-1}(K^1)\hookrightarrow r^\ast E$ being the inclusion map and 
$\tilde{H}(\cdot,1):p^{-1}(K^1)\rightarrow p^{-1}(x_0)\subset r^\ast E$.
Note that $\tilde H (\cdot,\cdot)|_{p^{-1}(x)\times [0,1]}
:p^{-1}(x)\times [0,1]\rightarrow p^{-1}(x_0)\subset r^\ast E$ is a homotopy equivalence for each $x\in K^1$.

Denote the space of self homotopy equivalences of $M$ by $G(M)$.
According to \cite{g65}, we have
$$
\pi_n(G(M))=\left\{
\begin{aligned}
&Center(\pi_1(M)), \ \ \ \ & if\ n=1;  \\
&\{e\}, \ \ \ \ & if\ n\geq 2. 
\end{aligned}
\right.
$$
By assumption, we have $Center(\pi_1(M))=\{e\}$.
Hence we can inductively extend 
$\tilde{H}(\cdot,1):p^{-1}(K^1)\rightarrow p^{-1}(x_0)$
to successive skeletons of $K$ on each cell until we obtain a continuous map 
$q:r^\ast E\rightarrow p^{-1}(x_0)$ such that $q|_{p^{-1}(x)}:p^{-1}(x)\rightarrow p^{-1}(x_0)$ is a homotopy equivalence for each $x\in K$.
Therefore, $r^\ast E$ is homotopy equivalence to $M\times B$.
As $\pi_1(B)$ is amenable , $\|B\|=0$.
Using (\ref{2021052101}), we have
$$\|E\|=\|r^\ast E\|\leq \tbinom{\dim E}{\dim M}\|M\|\times\|B\|=0.$$
Therefore, $\|E\|=0$.
\end{proof}

\section{Proof of Theorem \ref{thm:all B}}
Theorem \ref{thm:all B} is proved by applying method developed in the proof of \cite[Theorem F]{bfj16}.
\begin{definition}
Let $p_1:E_1\rightarrow B$ and $p_2:E_2\rightarrow B$ be two fiber bundles.
They are called fiber homotopy equivalent if there exists fiber-preserving maps 
$\varphi:E_1\rightarrow E_2$ and 
$\psi:E_2\rightarrow E_1$
such that $\psi\circ\varphi\simeq Id_{E_1}$ and
$\varphi\circ\psi\simeq Id_{E_2}$ through fiber-preserving homotopies $H_1$ and $H_2$,
respectively.
\end{definition}
\begin{remark}
Let
$$
\text{Diff}(M)=\{f:M\rightarrow M\text{ diffeomorphism}\}.
$$
By \cite[Corollary 7.6]{dl59}, 
smooth fiber bundle $M\rightarrow E\rightarrow B$
is homotopy equivalent to the pull-back bundled of some bundle $\tilde E\rightarrow \text{BDiff}(M)$ 
through some map $f:B\rightarrow \text{BDiff}(M)$ (which is called the classifying map).
There is a natural map c.
Two smooth fiber bundles is fiber homotopy equivalent if and only if their classifying map composing with 
$\Phi:\text{BDiff}(M)\rightarrow BG(M)$ is homotopy equivalent.
\end{remark}
\begin{definition}
Let $G$ be a groups. 
Let $Aut(G)$ be the automorphism groups of $G$ and $Inn(G)$ be the subgroup of $Aut(G)$ consisting of inner automorphisms.
The outer automorphism group of $G$ is the quotient group
$Aut(G)/Inn(G)$, denoted by $Out(G)$.
\end{definition}
\begin{remark}\label{re:out}
Note that the $0$-th homotopy group $\pi_0(G(M))$ is the set of homotopy classes of self homotopy equivalences of $M$.
Denote the projection by $\rho:G(M)\rightarrow \pi_0(G(M))$.
For an Eilenberg–MacLane CW-complex $K(\pi,1)$, it is well-known that there exists a isomorphism
$\eta:\pi_0(G(K(\pi,1)))\rightarrow Out(\pi)$(see \cite[Section 4.A]{h02} for more detail).
\end{remark}
Now, let $G$ be a topology group.
According to \cite{dl59}, there is a universal principal $G$-bundle
\begin{equation*}
    \begin{CD}
    G @>>>E_G\\
    @. @VVV \\
    @. B_G,
    \end{CD}
\end{equation*}
where $\pi_i(E_G)=0$ for all $i\in\mathbb N$.
Hence $G$ acts on $E_G$ and $B_G=E_G/G$.
Let $H<G$ be a subgroup.
Then the action of $G$ on $E_G$ can be restricted to an action of $H$ on $E_G$.
Let $H_1<H_2<G$ be two subgroup.
Then inclusion $H\hookrightarrow G$ induces a nature map $E_G/H_1\rightarrow E_G/H_2$
such that 
\begin{equation*}
    \begin{CD}
    H_2/H_1 @>>> E_G/H_1\\
    @. @VVV\\
    @.       E_G/H_2
    \end{CD}
\end{equation*}
is a topological fiber bundle.
A homomorphism  $\gamma:G\rightarrow G'$ between two topology groups induces map between fiber bundles
\begin{equation*}
    \begin{CD}
    E_G @>\tilde{\gamma}>> E_{G'}\\
    @VVV @VVV\\
    B_G @>\bar{\gamma}>> B_{G'}.
    \end{CD}
\end{equation*}
Hence we have the following commutative diagram
\begin{equation*}
\begin{tikzcd}
\cdots \arrow[r] & \pi_k(G) \arrow[r] \arrow[d, "\pi_k(\gamma)"] & \pi_k(E_G) \arrow[r] \arrow[d, "\pi_k(\tilde \gamma)"] & \pi_k(B_G) \arrow[r] \arrow[d, "\pi_k(\bar\gamma)"] & \cdots \arrow[r] & \pi_0(E_G) \arrow[d, "\pi_0(\tilde\gamma)"] \\
\cdots \arrow[r] & \pi_k(G') \arrow[r]                           & \pi_k(E_{G'}) \arrow[r]                                & \pi_k(B_{G'}) \arrow[r]                             & \cdots \arrow[r] & \pi_0(E_{G'}).                              
\end{tikzcd}
\end{equation*}
Hence if $\gamma$ is a weak homotopy equivalence, so is $\bar\gamma$.

We list \cite[Proposition 4.22]{h02} down below.
\begin{lemma}[Hatcher]\label{lem:hatcher}
Let $X$ be a CW-complex. 
Then a weak homotopy equivalence $Y\rightarrow Z$ induces bijection $[X,Y]\rightarrow [X,Z]$,
where the notation $[\cdot,\cdot]$ stands for the set of homotopy classes of maps between two spaces.
\end{lemma}
\begin{proof}[Proof of Theorem \ref{thm:all B}]
Let 
$$
\text{Diff}_0(M)=\{f\in\text{Diff}(M)|f\text{ is homotopic to }id_M\}.
$$
As 
\[
\text{Diff}_0(M)\hookrightarrow\text{Diff}(M)\hookrightarrow G(M),
\]
we can define
\[
\text{BDiff}_0(M)=E_{G(M)}/\text{Diff}_0(M)
\]
and
\[
\text{BDiff}(M)=E_{G(M)}/\text{Diff}(M).
\]
Hence we have a topological fiber bundle
\begin{equation*}
    \begin{CD}
\text{Diff}(M)/\text{Diff}_0(M) @>>>\text{BDiff}_0(M)\\
@. @VVV\\
@. \text{BDiff}(M)
    \end{CD}
\end{equation*}
and a induced map $\pi_{\text{BDiff}(M)}:\text{BDiff}(M)\rightarrow B_{G(M)}$.
According to Whitney’s approximation theorem \cite{w34} and Remark \ref{re:out}, there is a one-to-one map
\[
\text{Diff}(M)/\text{Diff}_0(M)\rightarrow \pi_0(G(M))\simeq Out(\pi_1(M)).
\]
By assumption $Out(\pi_1(M))$ is finite, hence the fiber $\text{Diff}(M)/\text{Diff}_0(M)$ is finite.
Therefore $\text{BDiff}_0(M)\rightarrow\text{BDiff}(M)$ is a finite sheet covering map.

Let $g:B\rightarrow \text{BDiff}(M)$ be the classifying map of
$q:E\rightarrow B$ (i.e. $E$ is the pull-back bundle of a bundle $\tilde E\rightarrow \text{BDiff}(M)$ through $g$).
Consider the following pull-back bundle
\begin{equation}\label{cd:B'}
\begin{CD}
B'@>\pi_{B'}>> \text{BDiff}_0(M)\\
@V\theta VV @V\pi_{\text{BDiff}_0(M)}VV\\
B@>g>> \text{BDiff}(M).
\end{CD}
\end{equation}
Hence $\theta:B'\rightarrow B$ is a finite sheet covering.

Consider the following pull-back bundle
\begin{equation*}
    \begin{CD}
    \theta^\ast E @>>>       E\\
    @VVV                     @VVV\\
    B'      @>\theta>> B.
    \end{CD}
\end{equation*}
Note that we have the commutative diagram (\ref{cd:B'}).
Hence the pull-back bundle
\begin{equation*}
\begin{tikzcd}
(\pi_{\text{BDiff}_0(M)}\circ\pi_{B'})^\ast\tilde E \arrow[d] \arrow[rr] &                                                        & \tilde E \arrow[d] \\
B' \arrow[r, "\pi_{B'}"]                                                 & \text{BDiff}_0(M) \arrow[r, "\pi_{\text{BDiff}_0(M)}"] & \text{BDiff}(M)   
\end{tikzcd}
\end{equation*}
is the same as the pull-back bundle
\begin{equation*}
\begin{tikzcd}
\theta^\ast E \arrow[d] \arrow[r] & E \arrow[r] \arrow[d] & \tilde E \arrow[d] \\
B' \arrow[r, "\theta"]                          & B \arrow[r, "g"]      & \text{BDiff}(M).  
\end{tikzcd}
\end{equation*}

Notice the composition map  
$$
\text{Diff}_0(M)\hookrightarrow \text{Diff}(M)\hookrightarrow G(M)\rightarrow \pi_0(G(M))
$$
is a constant map. 
Hence map $\text{BDiff}_0(M)\rightarrow B_{\pi_0(G(M))}$
factors through $E_{\pi_0(G(M))}$.
Notice that in the proof of \cite[Theorem 3.5]{dl59}, Dold and Lashof mentioned that a continous map from a compact space to $E_{\pi_0(G(M))}$ is null-homotopic.
Hence $\pi_{B'}:B'\rightarrow E_{G(M)}$ is null-homotopic.
Therefore, the map
\[
B'\rightarrow \text{BDiff}_0(M)\rightarrow\text{BDiff}(M)\rightarrow B_{G(M)}\rightarrow B_{\pi_0(G(M))}
\]
is null-homotopic.
Notice that $G(M)\rightarrow\pi_0(G(M))$ is weak homotopy equivalence.
Hence
$$
B_{G(M)}\rightarrow B_{\pi_0(G(M))}
$$
is a weak homotopy equivalence.
According to Lemma \ref{lem:hatcher}, $\theta^\ast E$ is fiber homotopy equivalence to $M\times B$.
Hence
\[
\|\theta^\ast E\|=\|M\times B\|
\]
By \ref{2021052101}, we have $\|\theta^\ast E\|=0$ if and only if $\|M\|\cdot\|B\|=0$.
Notice that $\theta^\ast E\rightarrow E$ is a finite sheet covering map.
Hence $\|E\|=0$ if and only if $\|M\|\cdot\|B\|=0$.
\end{proof}


\begin{thebibliography}{1111}
\bibitem[Buc09]{bucher2009simplicial} M. Bucher: Simplicial volume of products and fiber bundles. Discrete groups and geometric structures, 79–86, Contemp. Math. \textbf{501}, Amer. Math. Soc., Providence, RI, (2009). MR2581916
\bibitem[BFJ16]{bfj16} M. Bustamante, F. T. Farrell, Y. Jiang: Rigidity and characteristic classes of smooth bundles with nonpositively curved fibers, J. Topol. \textbf{9} (2016), 934–956. MR3551844
\bibitem[Cai35]{c35} S. S. Cairns. Triangulation of the manifold of class one. Bull. Amer. Math. Soc. \textbf{41} (1935), 549–552. MR1563139
\bibitem[DL59]{dl59} A. Dold, R. Lashof: Principal quasi-fibrations and fibre homotopy equivalence of bundles. Illinois J. Math. \textbf{3} (1959), 285–305.  MR0101521
\bibitem[FG16]{fg16} F. T. Farrell, A. Gogolev: On bundles that admit fiberwise hyperbolic dynamics, Math. Ann. \textbf{364} (2016), no. 1–2, 401–438. MR3451392
\bibitem[Got65]{g65} D. H. Gottlieb: A certain subgroup of the fundamental group. Amer. J. Math. \textbf{87} (1965), 840–856. MR0189027
\bibitem[Gro82]{gromov1982volume} M. Gromov: Volume and bounded cohomology. Inst. Hautes Études Sci. Publ. Math. \textbf{56} (1982), 5-99. MR0686042 
\bibitem[Gro87]{g87} M. Gromov: Hyperbolic groups. Essays in Group Theory (Mathematical Science Research Institute Publications, 8). Springer, New York, 1987, pp. 75–263. MR0919829
\bibitem[Hat02]{h02} A. Hatcher: Algebraic Topology. Cambridge University Press, Cambridge, 2002. MR1867354
\bibitem[Hos01]{hoster2001simplicial} M. Hoster, D. Kotschick: On the simplicial volumes of fiber bundles. Proc. Am. Math. Soc. \textbf{129(4)} (2001), 1229–1232. MR1709754
\bibitem[KR21]{kr21} T. Kastenholz, J. Reinhold: Essentiality and simplicial volume of manifolds fibered over spheres. arXiv:2107.05892 (2021).
\bibitem[LY72]{ly72} H. B. Lawson, S. T. Yau: Compact manifolds of nonpositive curvature. J. Differential Geometry \textbf{7} (1972), 211–228. MR0334083
\bibitem[LM21]{lm21} C. Loeh, M. Moraschini: Topological volumes of fibrations: A note on open covers. arXiv:2104.06038 (2021).
\bibitem[Mil56]{m56} J. Milnor: Construction of universal bundles. II, Annals of Mathematics \textbf{63} (1956), 430–436. MR0077932
\bibitem[Mos73]{m73} G. Mostow: Strong Rigidity of Locally Symmetric Spaces. Annals of Mathematical Studies, vol. 78. Princeton University Press, Princeton (1973) MR0385004
\bibitem[Thu78]{thurston1978geometry} W. Thurston: The geometry and topology of 3-manifolds. Lecture notes, Princeton (1978)
\bibitem[Whi34]{w34} H. Whitney: Analytic extensions of differentiable functions defined in closed sets, Trans. Am. Math. Soc. \textbf{36} (1934), 63–89. MR1501735
\bibitem[Whi40]{w40} J. H. C. Whitehead: On $C^1$-complexes. Ann. Math. \textbf{41(4)} (1940), 809–824. MR0002545
\bibitem[Whi49]{w49} J. H. C. Whitehead: Combinatorial homotopy. I. Bull. Am. Math. Soc. \textbf{55} (1949), 213–245. MR0030759
\end{thebibliography}
\end{document}